\documentclass{article}
\usepackage{amssymb,amsmath,wasysym,amsthm,hyperref,xcolor,booktabs}

\newtheorem{theorem}{Theorem}
\newtheorem{lemma}[theorem]{Lemma}
\newtheorem{corollary}[theorem]{Corollary}
\newtheorem{conjecture}[theorem]{Conjecture}

\begin{document}
\vspace*{-2cm}

\Large
\begin{center}
Three-term arithmetic progressions of consecutive powerful numbers \\ 

\hspace{10pt}

\large
Wouter van Doorn \\

\hspace{10pt}

\end{center}

\hspace{10pt}

\normalsize

\vspace{-25pt}
\centerline{\bf Abstract}
We show that infinitely many three-term arithmetic progressions $N, N+d, N+2d$ of powerful numbers exist with $d = 2\sqrt{N} + 1$. We further conjecture that infinitely many of these progressions consist of three consecutive terms in the sequence of powerful numbers, which would answer a question of Erd\H{o}s in the negative.

\section{Introduction}
We say that a positive integer $n$ is \emph{powerful} if for every prime divisor $p$ of $n$ we have that $p^2$ divides $n$. A well-known conjecture by Erd\H{o}s, Mollin and Walsh \cite{Er76d, MW} concerning powerful numbers states that there do not exist three consecutive integers $n, n+1, n+2$ which are all powerful. Although Chan \cite{Chan2} and She \cite{She} did manage to prove various partial results in this direction, it is still open in general. A more modest question asks for powerful integers forming an arithmetic progression, and here substantially more is known. \\

For example, one can find arbitrarily long arithmetic progressions of powerful numbers, simply by letting both the starting value and the common difference be equal to the square of the least common multiple of the first $n$ positive integers. More interestingly, Bajpai, Bennett, and Chan \cite{BBC} proved that there are infinitely many four-term progressions of \emph{coprime} powerful numbers, while Bennett and Walsh \cite{BW} found such an example with `merely' $111$ digits. Alternatively, if one requires the terms of the progression to be close together, then Chan \cite{Chan} notably showed the existence of infinitely many three-term progressions $N, N+d, N+2d$ of powerful integers with $d \le 4\sqrt{N} + O(1)$, while the abc-conjecture implies the lower bound $d \gg_{\epsilon} N^{1/2 - \epsilon}$ for every $\epsilon > 0$. \\

In this note we modify Chan's construction, and show that $d \le 2\sqrt{N} + 1$ can be attained. Even though this improvement may seem marginal, it turns out that this reaches the threshold relevant for a different question of Erd\H{o}s, which concerns consecutive elements in the sequence of powerful numbers. More precisely, in \cite{Er76d} Erd\H{o}s asked the following question, which is now recorded as Erd\H{o}s Problem \#938 on Bloom's website \cite{Bloom938}: are there only finitely many three-term arithmetic progressions of consecutive powerful numbers? We conjecture that our construction in fact yields infinitely many such examples.

\section{Acknowledgments and AI tool disclosure}
We thank Thomas Bloom for maintaining his beautiful website and for personal discussions on this topic. \\

In the making of this paper, ChatGPT was used as a sounding board, for proofreading and for supplying numerical examples, most notably the calculation of the recurrence relation \eqref{eq:recur} in Section \ref{constr} and Table \ref{tab:consecutive-powerful-ap} in Section \ref{data}. The paper itself was entirely human-generated.

\section{A Pellian construction} \label{pell}
\subsection{Three-term progressions with small common difference}
We start off with a sharpening of the unconditional upper bound of \cite[Theorem 7]{Chan}.

\begin{theorem}
There are infinitely many three-term arithmetic progressions $$N, N+d, N+2d$$ of powerful numbers with $d = 2\sqrt{N} + 1$.
\end{theorem}

\begin{proof}
Let $(x,y) \in \mathbb{N}^2$ be any solution to the generalized Pell equation

\begin{equation} \label{eq:x343y2}
x^2 - 7^3y^2 = 2.
\end{equation}

Then the three numbers 

\begin{equation} \label{solutions}
(x-2)^2, \qquad (x-1)^2, \qquad 7^3y^2 = x^2-2
\end{equation}

are all powerful and form the arithmetic progression $$N, N+d, N+2d$$ with $$N = (x-2)^2 \qquad \text{and} \qquad d = 2x-3 = 2\sqrt{N} + 1.$$ Since one can verify that equation \eqref{eq:x343y2} has infinitely many solutions (in fact, we will do so in the next section), the proof is finished. 
\end{proof}

\subsection{Explicit construction} \label{constr}
To actually construct infinitely many solutions to equation \eqref{eq:x343y2}, we first solve 

\begin{equation} \label{eq:x7y2}
x^2 - 7y^2 = 2,
\end{equation}

and then restrict to solutions with $7 | y$. Now, the equations $$x^2 - 7y^2 = 2 \qquad \text{and} \qquad x^2 - 7y^2 = 1$$ have specific solutions $(3, 1)$ and $(8, 3)$ respectively, so the theory of Pell equations gives the general (positive) solution $(x_k, y_k)$ to equation \eqref{eq:x7y2}, where $$x_k + y_k \sqrt{7} = (3 + \sqrt{7})(8 + 3 \sqrt{7})^k.$$ To get a recurrence relation out of this, we can multiply by $8 + 3 \sqrt{7}$ to get
\begin{align*}
x_{k+1} + y_{k+1} \sqrt{7} &= (x_k + y_k \sqrt{7})(8 + 3 \sqrt{7}) \\
&= (8x_k + 21y_k) + (3x_k + 8y_k)\sqrt{7}.
\end{align*}

Now, recall that we are only looking for solutions where $7$ divides $y$, so by reducing the recurrence relation modulo $7$ we see that $$x_{k+1} = 8x_k + 21y_k \equiv x_k \pmod{7},$$ which implies $$x_k \equiv x_0 \equiv 3 \pmod{7}$$ for all $k \ge 0$. Substituting this into the recurrence for $y_k$ gives $$y_{k+1} = 3x_k + 8y_k \equiv y_k + 2 \pmod{7},$$ so that we obtain $$y_k \equiv y_0 + 2k \equiv 1 + 2k \pmod{7}$$ for all $k \ge 0$. Hence, $7 | y_k$ precisely when $k \equiv 3 \pmod{7}$. In other words, $(x_k, y_k)$ solves equation \eqref{eq:x343y2} if 
\begin{align*}
x_k + 7y_k \sqrt{7} &= (3 + \sqrt{7})(8 + 3 \sqrt{7})^3\big((8 + 3 \sqrt{7})^7\big)^k \\
&= (11427 + 7 \cdot 617\sqrt{7})\big((8 + 3 \sqrt{7})^7\big)^k.
\end{align*}

Repeating our earlier calculation with $(8 + 3 \sqrt{7})^7$ instead of $(8 + 3 \sqrt{7})$ gives the recurrence $(x_{k+1}, y_{k+1}) = (ax_k + by_k, cx_k + dy_k)$ for the solutions to equation \eqref{eq:x343y2}, with 
\begin{align*}
a &= 130576328, \\
b &= 2418307437, \\
c &= 7050459, \\
d &= 130576328.
\end{align*}

Finally, by eliminating $y_k$ from this recurrence, one finds that $x_k$ satisfies the linear recurrence
\begin{align}
x_0 &= 11427, \nonumber \\
x_1 &= 2984191388685, \nonumber \\
x_{k+2} &= 261152656x_{k+1} - x_k. \label{eq:recur}
\end{align}

Moreover, with initial values $$y_0 = 617, \qquad \text{and} \qquad y_1 = 161131189369,$$ the sequence of $y_k$ satisfies the same recurrence relation \eqref{eq:recur}.

\section{But are they consecutive?} \label{consec}
\subsection{A Diophantine condition}
The three integers in \eqref{solutions} are \emph{consecutive} powerful numbers if these are the only powerful numbers in the half-open interval\footnote{Whilst this is correct as stated, it is not quite an equivalence, as it could be the case that $x^2 - 1$ happens to be powerful as well. In this case, however, we would even have three consecutive integers which are all powerful, solving the Erd\H{o}s–Mollin–Walsh conjecture!}\,$\big[(x-2)^2, x^2\big)$. How often this happens for general $x \in \mathbb{N}$ is exactly the object of study in \cite{NT}, and the crux is that the density of $x \in \mathbb{N}$ for which the open intervals $$\big((x-2)^2, (x-1)^2\big) \qquad \text{and} \qquad \big((x-1)^2, x^2\big)$$ contain exactly zero and one powerful numbers respectively\footnote{In \cite{NT} this density is denoted by $d(\mathcal{A}_{0,1}^{(2)})$.} is roughly $0.108$. The naive heuristic (which considers an $x$ satisfying equation \eqref{eq:x343y2} as a typical integer) therefore predicts that there are infinitely many triples of the form \eqref{solutions} that consist of consecutive powerful numbers. We can be somewhat less naive, however. \\

As pointed out to the author by Thomas Bloom, the results in \cite{NT} are based on the following equivalence, where $\{\xi\}$ denotes the fractional part of a real number $\xi$.

\begin{lemma}[cf. {\cite[Lemma 3]{NT}}] \label{dio}
For all integers $x \ge 3$, the number of powerful integers $a \in \big((x-2)^2, x^2\big)$ is equal to the number of squarefree $m \in \mathbb{N}$ for which

\begin{equation} \label{eq:mdio}
\left\{\frac{x}{m^{3/2}} \right\} < \frac{2}{m^{3/2}}.
\end{equation}
\end{lemma}

\begin{proof}
For a fixed squarefree $m \in \mathbb{N}$, if inequality \eqref{eq:mdio} holds, then there is a unique positive integer $l$ with $$\frac{x-2}{m^{3/2}} < l < \frac{x}{m^{3/2}}.$$ Multiplying by $m^{3/2}$ and squaring then gives us $$(x-2)^2 < l^2m^3 < x^2,$$ so that the powerful number $a = l^2m^3$ lies in the interval $\big((x-2)^2, x^2\big)$. \\

Conversely, any powerful number $a$ can be uniquely written as $a = l^2m^3$ for some $l, m \in \mathbb{N}$ with $m$ squarefree. And if we have $$(x-2)^2 < l^2m^3 < x^2,$$ then we see that inequality \eqref{eq:mdio} holds by taking the square root and dividing by $m^{3/2}$. 
\end{proof}

In our case $m = 1$ in Lemma \ref{dio} corresponds to $(x-1)^2$ and $m = 7$ corresponds to $x^2 - 2 = 7^3 y^2$. We therefore have the following corollary, by observing that equality in \eqref{eq:mdio} cannot occur for squarefree $m$ and $x \ge 3$.

\begin{corollary} \label{infiniteiff}
There are infinitely many $x$ for which the integers in \eqref{solutions} are consecutive powerful integers if there are infinitely many $x_k$ generated by the recurrence relation \eqref{eq:recur} for which the inequality

\begin{equation} \label{eq:mdioreverse}
\left\{\frac{x_k}{m^{3/2}} \right\} > \frac{2}{m^{3/2}}
\end{equation}

holds for all squarefree $m \in \mathbb{N} \setminus \{1, 7\}$.
\end{corollary}

\subsection{Provability} \label{provable}
In order to apply Corollary \ref{infiniteiff} we need to get a hold on the fractional parts of $x_k$. Here we have the following modest result.

\begin{theorem} \label{weakstuff}
For every fixed squarefree $m \ge 648560$ there are infinitely many $k \in \mathbb{N}$ such that inequality \eqref{eq:mdioreverse} holds.
\end{theorem}

\begin{proof}
Let $m \ge 648560$ be a fixed squarefree integer, and define 

\begin{equation} \label{char}
R(t) := t^2 - 261152656t + 1,
\end{equation}

which is the characteristic equation for the recurrence \eqref{eq:recur}. Writing $A := 130576328$, the roots of $R$ are then given by $$\alpha := A + \sqrt{A^2-1} \qquad \text{and} \qquad \alpha^{-1}.$$ With $x_k = B\alpha^k + C\alpha^{-k}$, we can calculate $B$ and $C$ by the equations
\begin{align*}
x_0 &= B + C, \\
x_1 &= B\alpha + C\alpha^{-1},
\end{align*}

which gives $$x_k = \left(\frac{x_1\alpha - x_0}{\alpha^2 - 1}\right) \alpha^k + \left(\frac{x_0\alpha^2 - x_1 \alpha}{\alpha^2-1}\right) \alpha^{-k}$$ for all $k \ge 0$. By applying \cite[Theorem 1.3]{CYZ} with 
\begin{align*}
\alpha_1 &= \alpha, \\
\alpha_2 &= \alpha^{-1}, \\
F_1(k) &= \frac{x_1\alpha - x_0}{m^{3/2}(\alpha^2 - 1)}, \\
F_2(k) &= \frac{x_0\alpha^2 - x_1 \alpha}{m^{3/2}(\alpha^2-1}), \\
L(R) &= 261152658,
\end{align*}

and condition $(c')$, we may conclude that the set of limit values of $\left\{\frac{x_k}{m^{3/2}} \right\}$ is not contained within an interval of length smaller than $\frac{1}{L(R)} > \frac{2}{m^{3/2}}$, where the latter inequality follows from the assumption $m \ge 648560$.
\end{proof}

Now, of course we need infinitely many $k$ for which inequality \eqref{eq:mdioreverse} holds for \emph{all} $m \notin \{1, 7\}$, which Theorem \ref{weakstuff} is a far cry away from. Let us therefore briefly move away from what we can prove at the moment, to what we are tempted to believe.

\subsection{Heuristics}
With Corollary \ref{infiniteiff} we can update our naive heuristic from before. For a generic $x$ and a fixed squarefree $m \notin \{1, 7\}$, the probability that inequality \eqref{eq:mdioreverse} holds is $1 - \frac{2}{m^{3/2}}$. Hence, the density of $x$ for which one expects inequality \eqref{eq:mdioreverse} to hold for all squarefree $m \notin \{1, 7\}$ is equal to $$\prod_{\substack{m \notin \{1, 7\} \\ m \text{ squarefree}}} \left(1 - \frac{2}{m^{3/2}} \right) \approx 0.055.$$ Such a heuristic can be made formal (see \cite{Shiu} and \cite{NT}), if we let $x$ run over all positive integers. In our case however, we are only interested in the $x_k$ as defined by recurrence \eqref{eq:recur}. For this sequence, using the characteristic equation \eqref{char} one can verify that the number of $x_k^2 \le n$ is asymptotically $$\frac{\log n}{2\log \left(A+\sqrt{A^2-1} \right)}$$ where $A = 130576328$. Hence, assuming that these $x_k$ behave in this regard like generic integers, we arrive at the following conjecture.

\begin{conjecture} \label{optimism}
The number of arithmetic progressions $N, N+d, N+2d$ of consecutive powerful numbers in the interval $[1, n]$ which are of the form $$(x-2)^2, \qquad (x-1)^2, \qquad 7^3y^2 = x^2-2$$ is equal to $\big(C_7 + o(1)\big)\log n$, where $$C_7 = \frac{1}{2} \prod_{\substack{m \notin \{1, 7\} \\ m \text{ squarefree}}} \left(1 - \frac{2}{m^{3/2}} \right) \Big(\log \left(A+\sqrt{A^2-1} \right)\Big)^{-1} \approx 0.0014.$$
\end{conjecture}

Of course, similar conjectures can be put forth for different squarefree integers in place of $7$. The reason we singled out the latter is that, with $M$ the set of all squarefree $m$ for which $x^2 - m^3y^2 = 2$ is solvable, it just so happens that $7$ is the smallest element of $M$. \\

In general it is unfortunately not obvious at a glance whether a given $m$ belongs to $M$, although \cite[Theorem 2]{Moll} states that this is equivalent to the congruence $u_0 \equiv 1 \pmod{m^3}$, where $(u_0, v_0)$ is the fundamental solution to $x^2 - m^3y^2 = 1$. Generalizing $C_7$ from Conjecture \ref{optimism} to $C_m$ for any $m \in M$, it would be interesting to decide whether $\sum_{m \in M} C_m$ converges.

\section{Alternative routes} \label{alt}
\subsection{A partial converse}
Two of the three elements of the arithmetic progression that we constructed in Section \ref{pell} are squares. As it turns out, in the present context this latter property uniquely determines our family of solutions. In order to precisely state what we mean by this, define $\mathcal{A}$ to be the set of all $N \in \mathbb{N}$ for which a $d \in \mathbb{N}$ exists such that $N, N+d, N+2d$ is a three-term arithmetic progression of consecutive powerful numbers. Note that such a $d$, if it exists, must be unique. So any $N \in \mathcal{A}$ is accompanied by a corresponding $d$. We can therefore write $\mathcal{A}$ as the disjoint union

\begin{equation} \label{eq:union}
\mathcal{A} = \mathcal{A}_0 \sqcup \mathcal{A}_1 \sqcup \mathcal{A}_2,
\end{equation}

where $\mathcal{A}_i$ (for $0 \le i \le 2$) is the set of all $N$ for which the corresponding triple $(N, N+d, N+2d)$ contains exactly $i$ squares. Note that the union in equation \eqref{eq:union} does indeed equal the full set $\mathcal{A}$, as three consecutive squares are never in arithmetic progression. We then have the following partial converse to the construction from Section \ref{pell}.

\begin{lemma} \label{iftwosquares}
For any $N \in \mathcal{A}$ we have $N \in \mathcal{A}_2$ if, and only if, for some $x \ge 3$ we have 

\begin{equation} \label{twosquares}
N = (x-2)^2, \qquad N+d = (x-1)^2, \qquad N+2d = x^2-2.
\end{equation}
\end{lemma}

\begin{proof}
As $N, N+d, N+2d$ are assumed to be consecutive elements in the sequence of powerful numbers, there cannot be a square other than possibly $N+d$ in the interval $(N, N+2d)$. This implies that if two of them are squares, then they must be consecutive squares. \\

We therefore have three options: either $N$ and $N+d$ are consecutive squares, or $N$ and $N+2d$ are consecutive squares, or $N+d$ and $N+2d$ are consecutive squares. The first case reduces to \eqref{twosquares}, while the second case is impossible as the difference between consecutive squares is never even. Finally, in the third case let us write $N+d = x^2$. Then we find $d = 2x+1$, in which case $$N = x^2 - 2x - 1 < (x-1)^2 < N+d,$$ contradicting the non-existence of other squares in the interval $(N, N+2d)$.
\end{proof}

\subsection{More heuristics}
Of course equation \eqref{eq:union} begs the question which, if any, of the sets $\mathcal{A}_0, \mathcal{A}_1, \mathcal{A}_2$ are non-empty or even infinite. Conjecture \ref{optimism} predicts that $$|\mathcal{A}_2 \cap \{1, 2, \ldots, n\}| \gg \log n,$$ while Lemma \ref{iftwosquares} shows that all $N \in \mathcal{A}_2$ must come from Pell equations. In fact, with $M$ and $C_m$ as defined in the discussion following Conjecture \ref{optimism}, if $\sum_{m \in M} C_m$ converges, one expects $$|\mathcal{A}_2 \cap \{1, 2, \ldots, n\}| \asymp \log n.$$ As for $\mathcal{A}_0$ and $\mathcal{A}_1$, recall that for a positive density of $x \in \mathbb{N}$ there are exactly three powerful integers $a, b, c$ with $(x-1)^2 < a < b < c < x^2$. The naive heuristic would view $a$, $b$ and $c$ as chosen randomly from the interval $\big((x-1)^2, x^2\big)$, which makes the probability that $a, b, c$ are in arithmetic progression $\asymp \frac{1}{x}$. Similarly, the probability that $b, c, x^2$ are in arithmetic progression is also $\asymp \frac{1}{x}$. As the density of $x \in \mathbb{N}$ for which there are exactly $k$ powerful integers in the interval $\big((x-1)^2, x^2\big)$ goes to $0$ with $k$ superpolynomially fast (see e.g. \cite[equation $(6)$]{NT}), they are not expected to influence the asymptotic. Summing the quantity $\asymp \frac{1}{x}$ over all $x$ leads to logarithmic growth as well, which tempts us to posit the following conjecture.

\begin{conjecture} \label{moreoptimism}
For all $i \in \{0, 1, 2\}$ we have $$|\mathcal{A}_i \cap \{1, 2, \ldots, n\}| \asymp \log n.$$ In particular, $$|\mathcal{A} \cap \{1, 2, \ldots, n\}| \asymp \log n.$$
\end{conjecture}

\subsection{Data} \label{data}
The theoretical evidence for Conjecture \ref{moreoptimism} is arguably somewhat flimsy, and the wish is certainly father to the thought in this matter. We therefore searched for examples as well, and below $10^{14}$ we managed to find $18$ triples of consecutive powerful numbers in arithmetic progression. That is, $$|\mathcal{A} \cap \{1, 2, \ldots, 10^{14}\}| = 18.$$ Perhaps surprisingly, every single one of these $18$ examples is actually an element of $\mathcal{A}_1$; see Table \ref{tab:consecutive-powerful-ap}.

\clearpage
\begin{table}[ht]
\centering
\begin{tabular}{rrrrr}
\toprule
$N$ & $N+d$ & $N+2d$ & $d$ & $d/\sqrt{N}$ \\
\midrule
1728 & 1764 & 1800 & 36 & 0.866 \\
6912 & 7056 & 7200 & 144 & 1.732 \\
729000 & 729316 & 729632 & 316 & 0.370 \\
1458000 & 1458632 & 1459264 & 632 & 0.523 \\
2916000 & 2917264 & 2918528 & 1264 & 0.740 \\
11664000 & 11669056 & 11674112 & 5056 & 1.480 \\
149022674775 & 149022848000 & 149023021225 & 173225 & 0.449 \\
260102040004 & 260102223752 & 260102407500 & 183748 & 0.360 \\
348796224200 & 348796548100 & 348796872000 & 323900 & 0.548 \\
697592448400 & 697593096200 & 697593744000 & 647800 & 0.776 \\
1040408160016 & 1040408895008 & 1040409630000 & 734992 & 0.721 \\
1206916971500 & 1206917268552 & 1206917565604 & 297052 & 0.270 \\
1395184896800 & 1395186192400 & 1395187488000 & 1295600 & 1.097 \\
2413833943000 & 2413834537104 & 2413835131208 & 594104 & 0.382 \\
4827667886000 & 4827669074208 & 4827670262416 & 1188208 & 0.541 \\
10862252743500 & 10862255416968 & 10862258090436 & 2673468 & 0.811 \\
21724505487000 & 21724510833936 & 21724516180872 & 5346936 & 1.147 \\
60345848575000 & 60345863427600 & 60345878280200 & 14852600 & 1.912 \\
\bottomrule
\end{tabular}
\caption{Three consecutive powerful integers in arithmetic progression found up to $10^{14}$.}
\label{tab:consecutive-powerful-ap}
\end{table}

It is perhaps noteworthy that in the above table it happens more than once that both $N, N+d, N+2d$ and $2N, 2N+2d, 2N+4d$ occur as triples. Of course, whether this happens infinitely often remains out of reach.

\end{document}